\documentclass[10pt]{amsart}
\usepackage[a4paper,margin=2.8cm]{geometry}

\usepackage{setspace}
\usepackage{amssymb}
\usepackage{amsmath}
\usepackage{amsthm}

\usepackage[shortlabels]{enumitem}
\usepackage{tikz}
\usepackage{arydshln}

\newtheorem{theorem}{Theorem}[section]
\newtheorem*{theorem*}{Theorem}
\newtheorem{proposition}[theorem]{Proposition}
\newtheorem{lemma}[theorem]{Lemma}
\newtheorem{corollary}[theorem]{Corollary}
\newtheorem{claim}[]{Claim}
\theoremstyle{definition}
\newtheorem{definition}[theorem]{Definition}

\newtheorem{problem}[theorem]{Problem}

\newtheorem{remark}[theorem]{Remark}

\newtheoremstyle{named}{}{}{\itshape}{}{\bfseries}{.}{.5em}{\thmnote{#3}}
\theoremstyle{named}

\usepackage[pdftex,unicode]{hyperref}   
\hypersetup{unicode}
\hypersetup{breaklinks=true}
\hypersetup{urlcolor=blue}

\newcommand{\keywordss}[1]{%
    \par\addvspace\baselineskip
    \noindent\textbf{Keywords:} #1%
}

\newcommand{\K}{\mathbb K}
\newcommand{\F}{\mathbb F}

\renewcommand{\P}{\mathbb P}

\newcommand{\U}{\underline}

\DeclareMathOperator{\GammaL}{\Gamma L}

\DeclareMathOperator{\AGL}{AGL}
\DeclareMathOperator{\diag}{diag}

\DeclareMathOperator{\QH}{QH}
\DeclareMathOperator{\SQH}{SQH}
\DeclareMathOperator{\PQH}{PQH}
\DeclareMathOperator{\PSQH}{PSQH}

\DeclareMathOperator{\End}{End}
\DeclareMathOperator{\GL}{GL}
\DeclareMathOperator{\Aut}{Aut}

\DeclareMathOperator{\id}{id}

\DeclareMathOperator{\lclm}{lclm}
\DeclareMathOperator{\gcrd}{gcrd}

\DeclareMathOperator{\CCZ}{CCZ}

\DeclareMathOperator{\DDT}{DDT}

\newcommand{\PN}[1]{{\mathbb{P}^{#1}}}

\newcommand{\parskipsize}{0.4em}
\usepackage{enumitem}
\setlist{  
  listparindent=\parindent,
  parsep=\parskip,
}

\begin{document}

\title{Triprojective almost perfect nonlinear permutations and functions}
\author{Faruk G\"{o}lo\u{g}lu}
\address{Charles University}
\curraddr{}
\email{faruk.gologlu@mff.cuni.cz}
\thanks{}
\author{Lukas K\"olsch}
\address{University of South Florida}
\curraddr{}
\email{lukas.koelsch.math@gmail.com}
\thanks{}
\subjclass[2020]{Primary XXXXX; Secondary XXXXX}
\date{}
\dedicatory{}

\begin{abstract}
We give a large family of almost perfect nonlinear (APN)
permutations of finite vector spaces of every odd 
dimension divisible by three. 
We also give APN functions that are not bijective on
even dimensions and related highly nonlinear functions.
The functions we provide admit a so-called triprojective 
structure induced by the general linear group $\GL(3,2^m)$.
\end{abstract}

\maketitle
\keywordss{Almost Perfect Nonlinear Functions, Substitution Permutations Networks, APN permutations}
\setcounter{tocdepth}{1}
\tableofcontents


\section{Introduction}
Let $F : \F_p^n \to \F_p^n$ be a vectorial function with $p$ a prime.
The difference distribution table of $F$ is defined as
\[
	\DDT_F(u,v) := \# \{ x \in \F_p^n : F(x+u) - F(x) = v \}.
\]
The differential uniformity of $F$ is 
\[
	\Delta_F := \max \{ \DDT_F(u,v) : u,v \in \F_p^n, u \ne 0\}.
\]
Functions with optimal differential uniformity $\Delta_F = 1$ are called 
perfect nonlinear (PN) which can happen only when $p$ is odd.
When $p = 2$, functions with optimal differential uniformity 
$\Delta_F = 2$ are called almost perfect nonlinear (APN). 
Examples in this case include the (multiplicative) 
field-inversion function when $n$ is odd. The notion arises from 
cryptography in the context of differential cryptanalysis.
In the cryptographic context of S-Boxes used in the frequently used 
substitution-permutation networks (for which AES is an example), 
permutations with low differential uniformity are employed (for which
the field-inverse function of $\F_{2^8}$ is the example used in 
AES). An important question is the existence question for APN 
bijections when $n > 6$ is even. We refer to the 
book~\cite{Carlet2021_Book} for details and connections.

\subsection{Notions of equivalence of vectorial functions}
Two Boolean functions $F,G : \F_p^n \to \F_p^n$ are said to
be linear-equivalent if they satisfy
\[
	L_1 \circ F \circ L_2 = G,
\]
for some $L_1,L_2 \in \GL(n,p)$, i.e., nonsingular linear 
transformations of the vector space $\F_p^n$. This corresponds
to basis change for the domain and co-domain vector-spaces. 
A coarser notion of equivalence is the extended-affine equivalence
which requires existence of affine bijections $A_1,A_2 \in \AGL(n,p)$
and an $\F_p$-affine function $A_3$ satisfying
\[
	A_1 \circ F \circ A_2 + A_3 = G.
\]
Still a coarser notion is the $\CCZ$-equivalence, which
requires existence of $n \times n$ $\F_p$-matrices 
$A,B,C,D \in M_n(\F_p)$ viewed as linear transformations in $\End(\F_p^n)$
forming a nonsingular matrix  in
$M_{2n}(\F_p)$ viewed as a bijection in $\GL(2n,p)$, and vectors $u,v \in \F_p^n$ that satisfy
\[
		\begin{pmatrix}
			A & B \\
			C & D
		\end{pmatrix}
		\begin{pmatrix}
			x \\
			F(x)
		\end{pmatrix}
		+
		\begin{pmatrix}
			u \\
			v
		\end{pmatrix}
		=
		\begin{pmatrix}
			\pi(x) \\
			G(\pi(x))
		\end{pmatrix}
\]
with $\pi : \F_p^n \to \F_p^n$ bijective. All of these notions 
of equivalence preserve being PN/APN by preserving
differential uniformity. $\CCZ$-equivalence contains both 
notions of equivalence which can be seen by
fixing $B = 0_{n\times n}$ and 
fixing $B = C = 0_{n\times n}, u = v = 0$, respectively. 

\subsection{Quadratic, semi-quadratic forms and multiprojective functions}
We write $F : \F_p^n \to \F_p^n$ as 
$F= (F_1,\ldots,F_n)$ where $F_i : \F_p^n \to \F_p$
for $1 \le i \le n$ and define
\[
	\deg F = \max \{ \deg F_i : 1 \le i \le n \},
\]
where $\deg F_i$ is the degree of the $F_i$ when written (uniquely) 
as a polynomial in 
$$\F_p[x_1,\ldots,x_n]/(x_1^p-x_1,\ldots,x_n^p-x_n).$$
If $F$ is quadratic, each $F_i$ for $1 \le i \le n$ can be written as
\[
	F_i := (x_1,\ldots,x_n) A^{(i)} (x_1,\ldots,x_n)^T +  
	\langle u^{(i)},(x_1,\ldots,x_n) \rangle + \epsilon^{(i)},
\]
for $n\times n$ matrices $A^{(i)} \in M_n(\F_p)$ with coefficients 
in $\F_p$, $u^{(i)} \in \F_p^n$ and $\epsilon^{(i)} \in \F_p$,
and where $(x_1,\ldots,x_n)$ are $\F_p$-variables. 
Without loss of generality, we will take 
$u^{(i)} = \U 0$ and $\epsilon^{(i)}=0$, i.e.,
\[
F_i := (x_1,\ldots,x_n) A^{(i)} (x_1,\ldots,x_n)^T 
\]
since we are interested in PN/APN functions up to 
EA-equivalence. These are called  $\F_p$-quadratic forms
and one can take $A^{(i)}$ to be in upper triangular form
(when $p = 2$ we can take main diagonal to be entirely zeroes since
$x_j^2 A^{(i)}_{jj} \pmod{x_j^2 - x_j}$
absorbs into $\langle u^{(i)},\underline x \rangle$).

When $n = m\ell$, we can view $\F_p^n$ as 
$\F_{p^m}^\ell$ by fixing $\F_{p^m}$-bases over $\F_p$.
We can now consider a collection of $\F_{p^m}$-quadratic forms
(for $1 \le i \le \ell$)
\[
	G_i := (y_1,\ldots,y_\ell) B^{(i)} (y_1,\ldots,y_\ell)^T,
\]
for $\ell \times \ell$ matrices $B^{(i)} \in M_\ell(\F_{p^m})$ 
with coefficients in $\F_{p^m}$ and $(y_1,\ldots,y_\ell)$ are 
$\F_{p^m}$-variables. Now $G = (G_1,\ldots,G_\ell)$ is a map
$G : \F_{p^m}^\ell \to \F_{p^m}^\ell$ which can be viewed
as $\F_p^n \to \F_p^n$ by fixing $\F_{p}$-bases
of $\F_{p^m}$. We say that 
$G \in \QH(\ell,p^m)$ and 
$F \in \QH(n,   p)$.
Note that (fixing bases) we have $\QH(\ell,p^m) \subseteq \QH(n,p)$
and if $m > 1$ (and in turn $\ell < n$) the inclusion is strict.
Here, one introduces semi-quadratic forms
\[
    H_i := (y_1,\ldots,y_\ell) C^{(i)} 
	(y_1^{\sigma_i},\ldots,y_\ell^{\sigma_i})^T,
\]
where $\sigma_i \in \Aut(\F_{p^m})$, i.e., $\sigma_i = p^{e_i}$ 
for some $e_i \in \{0,\ldots,m-1\}$ for all $1 \le i \le \ell$.
The map $H = (H_1,\ldots,H_\ell)$ is called a
$(\sigma_1,\ldots,\sigma_\ell)$-$\F_{p^m}$-multiprojective map, 
and using the $\F_{p^m}$-base embedding again, 
$H : \F_{p^m}^\ell \to \F_{p^m}^\ell$
can be viewed as a quadratic map $\F_p^n \to \F_p^n$.
When $\sigma_i = \sigma = \sigma_j$ for all $1 \le i,j \le \ell$
we say that $H$ is a $\sigma$-semiquadratic map and write 
$H \in \SQH(\ell,p^m)$. We have
\[
\QH(\ell,p^m) \subseteq \SQH(\ell,p^m) \subseteq 
\textrm{$(\sigma_1,\ldots,\sigma_\ell)$-$\F_{p^m}$-multiprojective maps} \subseteq \QH(n,p).
\]

The reason one is interested in semi-quadratic and multiprojective maps
is that they still have some of the structure that $\QH(\ell,p^m)$ has
and they are larger and give more examples of PN/APN functions. In fact,
there are provably exponentially many pairwise inequivalent
PN/APN functions in the class of multiprojective maps.
Note that semiquadratic and multiprojective maps admit automorphisms
(with respect to the foregoing notions of equivalences)
\begin{equation} \label{eq:aut}
   H \circ
	\diag (a,\ldots, a)
	=  \diag(a^{1+\sigma_1}, \ldots, a^{1+\sigma_\ell}) 
	\circ H
\end{equation}

where $a \in \F_{p^m}^\times$ and $\diag(\underline u)$ is the diagonal 
matrix in $\GL(\ell,p^m)$ with the main diagonal $u_i \in \F_{p^m}$ 
and the rest $0$. Moreover, $\GL(\ell,p^m) \le \GL(n,p)$ acts on the right
on multiprojective maps and both on the left and on the right on 
semiquadratic maps, providing the subfield structure we mentioned before.
This structure has been used 
\cite{Gologlu22, Gologlu23, gologlu2025equivalences, golouglu2023exponential, SPKG} 
for solving problems in the theory of highly nonlinear functions.

\subsection{Bijections of $\PN{\ell-1}(\F_q)$}
Let $q = p^m$ and define an equivalence relation among vectors 
$\U u,\U v \in \F_q^\ell \setminus \{\U 0\}$
as $\U u \sim \U v$ if and only if $\U u = \lambda \U v$ 
for some $\lambda \in \F_q^\times$. 
Recall that the projective space $\PN{\ell-1}(\F_q)$ is defined as the set of 
equivalence classes of $\F_q^\ell \setminus \{\U 0\}$ modulo the
equivalence relation $\sim$. Any map $F : \F_q^\ell \to \F_q^\ell$
that satisfies $F(\lambda\U x) = \phi(\lambda)F(\U x)$ for 
$\lambda \in \F_q^\times$ and $\phi : \F_q \to \F_q$, 
and $F(\U x) \ne 0$ for $\U x \in \F_q^\ell \setminus \{\U 0\}$ induces a 
well-defined map 
$\overline{F}$ on the projective space $\PN{\ell-1}(\F_q)$. 

In \cite{comm_paper}, we defined projective versions $\PQH(\ell,q)$ 
and $\PSQH(\ell,q)$ of $\QH$ and $\SQH$ as follows.

\begin{definition} 
The projective homogeneous semiquadratic (resp. quadratic) transformations
$\overline{F} \in \PSQH(\ell,q)$ (resp. $\PQH(\ell,q)$) 
are transformations
$\overline{F} : \PN{\ell-1}(\F_q) \to \PN{\ell-1}(\F_q)$
induced by $F \in \SQH(\ell,q)$ (resp. $\QH(\ell,q)$) 
that satisfy $F(\U x) = \U 0 \iff \U x = \U 0$.
\end{definition}

The main result in \cite{comm_paper} is the following result 
on semi-quadratic bijections of arbitrary characteristic.

\begin{theorem} \label{thm:all_char_old_paper}
Let $q = p^m$ and $f,g,h \colon \F_q^3\rightarrow \F_q$ be defined as
\begin{align*}
f(x,y,z) &=  x^{\sigma+1} + ay^{\sigma}z + bx^{\sigma}y + cx^{\sigma}z,\\
g(x,y,z) &= ay^{\sigma+1} +  z^{\sigma}x + bz^{\sigma}y + cx^{\sigma}y,\\
h(x,y,z) &=  z^{\sigma+1} -  x^{\sigma}y,
\end{align*}
where $a,b,c \in \F_q$ and 
$\sigma=p^k$.

Then $\overline{F} = (\overline {f,g,h})\in \PSQH(3,q)$ is bijective 
if and only if the equation	
\begin{equation}
		ax^{\sigma^2+\sigma+1}+bx^{\sigma+1}+cx+1=0
\label{eq:condition_oldmain}
\end{equation}
has no solution $x \in \F_q$. 
\end{theorem}
Note that in~\cite{comm_paper}, Eq.~\eqref{eq:condition_oldmain} is written as $x^{\sigma^2+\sigma+1} + cx^{\sigma^2+\sigma} + bx^{\sigma^2} + a = 0$. As shown in~\cite[Lemma B.4.]{comm_paper}, these two conditions are equivalent.
\subsection{Main results} 
The following is the first main 
theorem of this paper, establishing the APN property of the functions
of Theorem \ref{thm:all_char_old_paper} in characteristic 2.

\begin{theorem} \label{thm_main}
Let $q = 2^m$ and $F_1,F_2,F_3 $ be defined as
\begin{align*}
F_1(x,y,z) &=  x^{\sigma+1} + ay^{\sigma}z + bx^{\sigma}y + cx^{\sigma}z,\\
F_2(x,y,z) &= ay^{\sigma+1} +  z^{\sigma}x + bz^{\sigma}y + cx^{\sigma}y,\\
F_3(x,y,z) &=  z^{\sigma+1} +  x^{\sigma}y,
\end{align*}
where $a,b,c \in \F_q$, $\sigma=2^k$, $1\leq k <m$, $d=\gcd(k,m)$ and $a\neq 0$. 
Define $F \in \SQH(3,q)$ via 
\[
F = ({F_1,F_2,F_3})\colon \F_q^3\rightarrow \F_q^3. 
\]

If the equation 
\begin{equation} \label{eq:condition_main}
    	ax^{\sigma^2+\sigma+1}+bx^{\sigma+1}+cx+1=0
\end{equation}
has no solution $x \in \F_q$ then the differential uniformity of $F$ is $2^d$.
\end{theorem}

In particular, by choosing $d=1$ in Theorem \ref{thm_main}, 
we construct APN functions. 

In \cite{comm_paper}, we proved the odd $p$ case and showed that the 
corresponding functions are PN. In this paper, we address even 
characteristic and APN functions. The tools to prove this result 
differ considerably.

For the sequel, we will need the following well-known lemma.

\begin{lemma} \label{lem:gcd}
    Let $d=\gcd(m,n)$ and $p$ be a prime. Then 
    \begin{enumerate}
        \item  $\gcd(p^m-1,p^n-1)=p^d-1$
        \item $\gcd(p^m+1,p^n-1)=\begin{cases}
        1, & \text{if } n/d \text{ is odd and } p \text{ is even,}\\
        2, & \text{if } n/d \text{ is odd and } p \text{ is odd,}\\
       p^d+1, & \text{if } n/d \text{ is even.}\end{cases}$
    \end{enumerate}
\end{lemma}

The following theorem determines the image set size of the 
functions in Theorem~\ref{thm_main}. It is proven in 
\cite{comm_paper} only in odd characteristic. The proof also holds 
in even characteristic. We include the simple proof for convenience.

\begin{theorem}\label{thm:sqh}
Let $F \in \SQH(\ell,q)$ with companion automorphism $\sigma \in \Aut(\F_q)$.
If $\overline{F}$ is bijective on $\PN{\ell-1}(\F_q)$ then $F$ is 
$\gcd(\sigma+1,q-1)$-to-$1$ on $\F_q^\ell \setminus \{\U 0\}$.
\end{theorem}
\begin{proof}
Suppose $F(\underline{x_1})=F(\underline{x_2})$ 
for $\underline{x_1},\underline{x_2} \in \F_q^\ell \setminus \{\U 0\}$. 
Then clearly 
$\overline{F}(\underline{x_1})=\overline{F}(\underline{x_2})$,
and since $\overline{F}$ is bijective on $\PN{n-1}(\F_q)$, 
we have $\underline{x_2}=\lambda \underline{x_1}$ for some 
$\lambda \in \F_q^\times$. So 
$F(\underline{x_2})=\lambda^{\sigma+1}F(\underline{x_1})=F(\underline{x_1})$ 
if and only if $\lambda^{\sigma+1}=1$, 
i.e., $\lambda$ is a $(\sigma+1)$-st root of unity in $\F_q^\times$, 
and the number of such roots of unity is exactly $\gcd(\sigma+1,q-1)$. 
\end{proof}

The following theorem is the next main result, 
precisely classifying the image set sizes of the functions we consider. 
In particular, we prove that our family yields APN permutations. 
\begin{theorem} \label{thm:imageset}
    Let $q = 2^m$ and $F_1,F_2,F_3 $ be defined as
\begin{align*}
F_1(x,y,z) &=  x^{\sigma+1} + ay^{\sigma}z + bx^{\sigma}y + cx^{\sigma}z,\\
F_2(x,y,z) &= ay^{\sigma+1} +  z^{\sigma}x + bz^{\sigma}y + cx^{\sigma}y,\\
F_3(x,y,z) &=  z^{\sigma+1} -  x^{\sigma}y,
\end{align*}
where $a,b,c \in \F_q$, $\sigma=2^k$, $1\leq k <m$, $d=\gcd(k,m)$ and $a\neq 0$. 
Let 
\[
F = ({F_1,F_2,F_3})\colon \F_q^3\rightarrow \F_q^3
\]
as in Theorem~\ref{thm_main}.
Assume further that
\begin{equation*}
    	ax^{\sigma^2+\sigma+1}+bx^{\sigma+1}+cx+1=0
\end{equation*}
has no solution $x \in \F_q$. Then $F$ is 
\begin{itemize}
    \item $(2^d+1)$-to-one on $\F_q^3\setminus \{\U 0\}$ if $m/d$ is even,
    \item bijective on $\F_q^3$ if $m/d$ is odd.
\end{itemize}
\end{theorem}
In particular, $F$ is an APN permutation if and only if $d=1$ and $m$ is odd.
\begin{proof}
   By Theorem~\ref{thm:all_char_old_paper}, $\overline{F}$ is bijective on $\P^2(\F_q)$. Then Theorem~\ref{thm:sqh} implies that $F$ is 
$\gcd(2^k+1,q-1)$-to-$1$ on $\F_q^3 \setminus \{\U 0\}$. The greatest common divisor is then evaluated using Lemma~\ref{lem:gcd}.
\end{proof}


\section{Proof of Theorem~\ref{thm_main}}

In this section, we will prove Theorem~\ref{thm_main}. 
We distinguish the cases $m/d$ even 
(which turns out to be the easier case) and $m/d$ odd. 
For the proofs, we will introduce the following definition.

\begin{definition}
Let $q=p^m$ and define for $F\in \SQH(\ell,q)$ and  any $\U a \in \F_q^\ell$
the function $L_{\U{a}}(\U{x}) \colon \F_{q}^\ell\rightarrow \F_q^\ell$ 
via $L_{\U a}(\U x)=F(\U x + \U a)-F(\U x)- F(\U a)$. We call the largest 
subfield $\K \subseteq \F_q$ such that $L_{\U a}$ is $\K$-linear 
for all $a \in \F_q^\ell \setminus \{\U 0\}$ its \emph{core}.
\end{definition}

The following is immediate.
\begin{proposition}\label{prop_core}
Let $F \in \SQH(\ell,p^m)$ with companion automorphism 
$\sigma\colon x\mapsto x^{p^k}$ and $d=\gcd(k,m)$.  
Then its core is $\F_{p^{d}}$. 
\end{proposition}

\subsection{The case $m/d$ is even}

As explained above, the following theorem 
combined with Theorem~\ref{thm:all_char_old_paper} proves 
Theorem~\ref{thm_main} for the case that $m/d$ is even. The proof 
in this case is essentially identical to the proof in the odd 
characteristic case given in~\cite{comm_paper}. We give the short proof 
for completeness.

\begin{theorem} \label{thm:unif}
Let $F \in \SQH(\ell,2^m)$ with core $\F_{2^d}$, $m/d$ even. 
If $\overline{F}$ is bijective on $\PN{\ell-1}(\F_{2^m})$ then $F$ is differentially 
$2^d$-uniform.
\end{theorem}
\begin{proof}
By Theorem~\ref{thm:sqh} and Lemma~\ref{lem:gcd}, $F$ is 
$(2^d+1)$-to-$1$ on $\F_{2^m}^\ell \setminus \{\U 0\}$, more precisely, 
$F(\U x) = F(\U y)$ if and only if $\U y = \lambda \U x$ for 
some $\lambda \in \F_{2^m}$ such that $\lambda^{2^d+1}=1$.
This means that for every fixed nonzero $\U a$,
the equation $D_{\U{a}}(\U x)=F(\U x+\U a)+F(\U x)=\U 0$ 
has exactly as many roots as there are solutions to 
$\U x+\U a=\lambda \U x$ where  $\lambda^{2^d+1}=1$. 
Since $\U a\neq \U 0$, there is no solution for $\lambda=1$; in all other cases rearranging yields $\U x=\U a/(\lambda-1)$. For each $\lambda \neq 1$ we thus get exactly one solution and $D_{\U a}(\U x)$ has exactly $2^d$ roots. Since $L_{\U a}=D_{\U a}+F(\U a)$ is linear over $\F_{2^d}$, the number of solutions of $L_{\U a}(\U x)=\U 0$ is thus either $0$ or $2^d$, in particular, $F$ is differentially $2^d$-uniform.
\end{proof}

\begin{proof}[Proof of Theorem~\ref{thm_main} if $m/d$ is even]
Let $F \in \SQH(3,2^m)$ be the function defined in Theorem~\ref{thm_main}. The core of $F$ is exactly $\F_{2^d}$ by Proposition~\ref{prop_core}, and by Theorem~\ref{thm:all_char_old_paper}, $\overline{F}$ is bijective on $\PN{2}(\F_{2^m})$. Then by Theorem~\ref{thm:unif}, $F$ has differential uniformity $2^d$. 
\end{proof}
\begin{remark}\label{rem:mdeven}\begin{itemize}
	\item Generally, Theorem~\ref{thm:unif} can be used to produce APN functions on $\F_{2}^{m\ell}$ as long as one can find $F \in \SQH(\ell,2^m)$ such that $\overline{F}$ is bijective on $\PN{\ell-1}(\F_{2^m})$. Finding and classifying such functions is thus particularly interesting. The case $\ell=2$ was completely classified by the first-named author~\cite{Gologlu22} as well as Ding-Zieve~\cite{DingZieve}. Theorem~\ref{thm:all_char_old_paper} gives an explicit infinite family for $\ell=3$ (which yields our family of APN functions in Theorem~\ref{thm_main} if $m/d$ is even). Finding further examples or classifications of such bijections of the projective space is an interesting research question.
	\item If $m/d$ is even, Theorem~\ref{thm_main} and Theorem~\ref{thm:imageset} show that the functions $F$ we are considering are (if $m/d$ is even) exactly $2^d$-uniform and $(2^d+1)$-to-$1$. Such functions were for instance studied in~\cite[Theorem 7]{kolsch2023image}, where the Walsh spectra of these functions are explicitly determined. 
	\end{itemize}
\end{remark}
\subsection{The case $m/d$ is odd} 

In this case, we need a completely new set of tools.

The polynomials 
\[
	a_n t^n + a_{n-1} t^{n-1} + \cdots + a_1 t + a_0,
\]
in the indeterminate $t$ with coefficients in $\F_q$, 
with the usual addition rule of polynomials and with 
the multiplication rule $ta=a^\sigma t$ for field 
automorphism $\sigma \in \Aut(\F_q)$ form a 
ring $\F_q[t;\sigma]$ called the \textit{twisted polynomial ring}.

Note that for $\sigma=\id$, we recover the usual polynomial ring, 
but for the other choices of the field automorphism, the multiplication 
in the ring of twisted polynomials is non-commutative. 
We refer to~\cite[Chapter 1]{jacobson2009finite} for an introduction 
to the twisted polynomial ring (or more generally, \emph{skew-polynomial rings}). 

We recall some basic facts for twisted polynomial rings over finite fields.
It is clear from the construction that for any $P,Q \in \F_q[t;\sigma]$ 
we have $\deg{PQ} = \deg{P}+\deg{Q}$.
We say $P \in \F_q[t;\sigma]$ is a right divisor of $Q \in \F_q[t;\sigma]$ 
if there is a polynomial $R \in \F_q[t;\sigma]$ such that $Q=R\cdot P$, 
and $Q$ is in this case a left multiple of $P$. 
We call $P$ a \emph{linear} right divisor if $\deg(P)=1$. 
We define (linear) left divisors as well as right multiples 
similarly. For any two twisted polynomials $P_1,P_2\in \F_q[t;\sigma]$, there is a unique monic polynomial $L \in \F_q[t;\sigma]$ such that $L$ is a left multiple of both $P_1$ and $P_2$, and any other common left multiple of $P_1$ and $P_2$ is a left multiple of $L$. 
Such $L$ is called the least common left multiple of $P_1$ and $P_2$, and we write $L=\lclm(P_1,P_2)$.
he degree of the least common left multiple satisfies 
$\deg \lclm(P_1,P_2) = \deg P_1 + \deg P_2 - \deg \gcrd(P_1,P_2)$
where the greatest common right divisor ($\gcrd$) is analogously defined.
The ring $\F_q[t;\sigma]$ for $\sigma \ne \id$ is not a unique 
factorization domain, however there are left and right Euclidean algorithms.
Thus, given $P,R \in \F_q[t;\sigma]$, there exist unique
$Q,S \in \F_q[t;\sigma]$ such that $P = Q \cdot R+S$, where $\deg(S)<\deg(R)$.
For the proofs we refer to \cite[Chapter 1]{jacobson2009finite}.

The following result by Ore establishes the connection of certain 
twisted polynomials with the family in Theorem~\ref{thm_main}.

\begin{theorem}{\cite[Theorem 3]{ore1933special}} \label{thm:ore}
The polynomial 
$P(X)=aX^{\sigma^2+\sigma+1}+bX^{\sigma+1}+cX+1 \in\F_q[X]$ 
has no roots in $\F_q$ 
if and only if 
$F(t)=at^3+bt^2+ct+1 \in \F_q[t;\sigma]$ has no linear right divisor.
\end{theorem}
The result of Ore is for arbitrary degree $F$. Here,
we recall just degree three which we use.
We also have the following result which links the left and right 
divisors also holds for arbitrary degrees. 

\begin{proposition}{\cite[Lemma B.2]{comm_paper}} \label{prop:divisor}
Let $P \in \F_q[t;\sigma]$ be a polynomial. 
Then $P$ has a linear right divisor if and only if 
     $P$ has a linear left divisor.
\end{proposition}

Theorem~\ref{thm:ore} in particular establishes a connection between degree $3$ twisted polynomials and the polynomials appearing in Theorem~\ref{thm_main}. We thus further investigate degree $3$ twisted polynomials.

\begin{proposition} \label{prop:lift}
Let $P \in \F_q[t;\sigma]$ with $\deg(P)=3$ such that $P$ has no 
linear right divisors. Then $P$ has also no linear right divisors 
in $\F_{q^2}[t;\sigma]$.
\end{proposition}
\begin{proof}
    Assume to the contrary that 
    \[P =P_2 \cdot (t-b)\]
    with $P_2 \in \F_{q^2}[t;\sigma]$ and $b\in\F_{q^2}\setminus \F_q$ (note that it is impossible that $b\in \F_q$ and $P_2 \notin \F_{q}[t;\sigma]$). Then also 
    \[P =P^{(q)}=P_2^{(q)} \cdot (t-b^q),\]
    where $P^{(q)}$ is the polynomial $P$ with every coefficient taken 
	to the $q$-th power. So $P$ is, in $\F_{q^2}[t;\sigma]$, a left multiple 
	of $(t-b)$ and $(t-b^q)$, so it is particular a left multiple of the 
	least common left multiple $R:=\lclm(t-b,t-b^q)$. We calculate $R$. 
	We have by the definition of the least common left multiple that
	$R$ is of degree two, and,
	\[R = (t-c_1)(t-b)=(t-c_2)(t-b^q)\]
	for some $c_1,c_2 \in \F_{q^2}$. Usual multiplication in the twisted polynomial ring yields
	\[R=t^2-(c_1+b^\sigma)t+c_1b=t^2-(c_2+(b^{q})^\sigma)t+c_2b^q.\]
	Comparing coefficients yields $c_1b=c_2b^q$ and $c_2+(b^q)^\sigma=c_1+b^\sigma$. Solving the second equation for $c_2$ and substituting in the first yields $c_1=b^q\frac{b^\sigma-(b^q)^\sigma}{b-b^q}$, so 
	$$R=t^2-\left(b^q\frac{b^\sigma-(b^q)^\sigma}{b-b^q}+b^q\right)t+b^{q+1}\frac{b^\sigma-(b^q)^\sigma}{b-b^q}=t^2-\left(\frac{b^{\sigma+1}-b^{q({\sigma+1})}}{b-b^q}\right)t+b^{q+1}\frac{b^\sigma-b^{\sigma q}}{b-b^q}.$$
	A routine check reveals that both the linear and the constant coefficients
	$d_1,d_0$ are 
	contained in $\F_q$ by simply computing $d_i^q-d_i=0$, so 
	$R \in \F_q[t;\sigma]$.
								
	Then $P = QR$ for some $Q \in \F_{q^2}[t;\sigma]$ with $\deg(Q)=1$. 
	Since $R \in \F_q[t;\sigma]$ we also have necessarily $Q \in \F_q[t;\sigma]$
	by the right Euclidean division algorithm.
	Thus, $P$ is left-divisible by a degree one polynomial in $\F_q[t;\sigma]$, 
	so it is also right-divisible by one by Proposition~\ref{prop:divisor}, 
	contradicting our assumption.
\end{proof}

\begin{corollary} \label{cor:roots}
    The polynomial $P(X)=aX^{\sigma^2+\sigma+1}+bX^{\sigma+1}+cX+1 \in\F_q[X]$ has no roots in $\F_q$ if and only if it has no roots in $\F_{q^2}$. 
\end{corollary}
\begin{proof}
    Let $s(t)=at^3+bt^2+ct+1 \in \F_q[t;\sigma]$.  
    Now assume $P$ has a root in $\F_{q^2}$. By Theorem~\ref{thm:ore}, $s$ has then a linear right divisor in $\F_{q^2}$, so by Proposition~\ref{prop:lift} it has a linear right divisor in $\F_q$, so $s$ has  (again by Theorem~\ref{thm:ore}) a root in $\F_q$. The converse is trivial.
\end{proof}

    We are now ready to prove Theorem~\ref{thm_main} for the case $m/d$ odd. The idea is to view $F \in \SQH(3,q)$ in a suitable extension field where we can then reduce the proof to the $m/d$ is even case we already proved.

\begin{proof}[Proof of Theorem~\ref{thm_main} if $m/d$ is odd.]
Let $q=2^m$, $F\in \SQH(3,q)$ be defined as in the Theorem~\ref{thm_main} with companion automorphism $\sigma=2^k$, $d=\gcd(k,m)$, $m/d$ is odd such that Eq.~\eqref{eq:condition_main} has no roots in $\F_q$.
    Let $k=2^\ell\cdot k'$, where $k'$ is odd, $\ell \geq 0$. 
    Consider $G\in \SQH(3,q^{2^{\ell+1}})$ defined exactly as in Theorem~\ref{thm_main}, using the same coefficients 
	$a,b,c \in \F_q \subseteq \F_{q^{2^{\ell+1}}}$ and the same field 
	automorphism $\sigma=2^k$. We can view $\F_{q}^3$ embedded naturally 
	in $\F_{q^{2^{\ell+1}}}^3$ and under this embedding, the restriction 
	of $G$ to $\F_{q}^3$ is exactly $F$. By repeated applications of 
	Corollary~\ref{cor:roots}, Eq.~\ref{eq:condition_main} has no roots in $\F_{q^{2^{\ell+1}}}$, so $G$ permutes $\P^2(\F_{q^{2^{\ell+1}}})$ by Theorem~\ref{thm:all_char_old_paper}. 
	Let $D = \gcd(2^{\ell+1}m, k)$. By the definition of $\ell$, the ratio $(2^{\ell+1}m)/D$ is even.  
	So by Theorem~\ref{thm:unif}, $G$ is differentially $2^D$ uniform,
	that is to say $|\ker (L_{G,\U y})| = 2^D$, where
	\[
		L_{G,\U y}(\U x)=G(\U x+\U y)+G(\U x)+G(\U y),
	\]
	for a fixed $\U y \in \F_{q^{2^{\ell+1}}}^3\setminus \{\U 0 \}$.
    
	We can explicitly determine $\ker (L_{G,\U y})$ by observing 
	(using Proposition \ref{prop_core}) that the core of $G$ is
	$\F_{2^D}$. 
	Indeed, $\U y \in \ker(L_{G,\U y})$ (since we work in characteristic $2$)
	and then 
	$\ker(L_{G,\U y}) = \U y \, \F_{2^D}$ 
	(since $L_{G,\U y}$ is $\F_{2^D}$-linear). 

    Consider now $L_{F,\U y}(\U x)=F(\U x+\U y)+F(\U x)+F(\U y)$ for any 
	fixed $\U y \in \F_q^3 \setminus \{\U 0 \}$. 
	We have $\ker(L_{F,\U y}) = \ker(L_{G,\U y}) \cap \F_q^3$.
    Because $\gcd(D, m) = \gcd(2^{\ell+1}m, k, m) = \gcd(k,m) = d$, we have $\F_{2^D}\cap \F_{q}=\F_{2^d}$. Thus, 
	$\ker(L_{G,\U y})\cap \F_q^3= \U y \, \F_{2^d}$. In particular, 
	$|\ker(L_{F,\U y})| = 2^d$, 
	and thus $F$ has differential uniformity $2^d$.
\end{proof}
    
\section{Comparison to existing trivariate APN functions} \label{sec:comparison}
Two families of APN permutations were found by Li and Kaleyski~\cite{li2023two}. One of these was recently extended by Bartoli and Stanica, who also gave one more family of trivariate APN functions~\cite{bartoli2026infinite}. We will now show that all of these are very special cases of our family of functions we presented in Theorem~\ref{thm_main}. We want to note that the proofs of the theorems in~\cite{li2023two} and~\cite{bartoli2026infinite} are spanning multiple pages of complicated case distinctions each, and some are computer assisted in the sense that computer algebra systems are used to compute complicated resultants. Others use deep tools from algebraic geometry, like the Hasse-Weil bounds. Our results are computer free, much shorter and purely elementary. We thus give a simple and unified explanation for all these functions. We also want to emphasize that the approach using resultants led to a number of unnecessary additional conditions in~\cite{li2023two}. They also spend multiple pages on proving the image set size of these functions in certain cases, we can recover and generalize these results easily with Theorem~\ref{thm:sqh} (whose proof only spans a few lines).

\begin{theorem}[{\cite[Theorem 1]{li2023two}}] \label{thm:li1}
    Let $q=2^m$, $\gcd(k,m)=d$, $\sigma=2^k$ and define the following function on $\F_q^3$:$$F(x,y,z)=(x^{\sigma+1}+x^{\sigma}z+yz^{\sigma},x^{\sigma}z+y^{\sigma +1},xy^{\sigma}+y^{\sigma}z+z^{\sigma+1}).$$
    Assume that the polynomials $X^{\sigma^2+\sigma+1}+X+1$, $X^{\sigma^2+\sigma+1}+X^{\sigma^2}+1$ and $X^{\sigma^2+\sigma+1}+X^{\sigma^2+1}+X^{\sigma+1}+X+1$ have no roots in $\F_{2^m}$ and $X^{\sigma^2+\sigma+1}+XY^{\sigma^2+\sigma}+XY^{\sigma}+X^{\sigma^2+\sigma}+X^{\sigma}Y^{\sigma^2}+X^{\sigma^2}Y+Y^{\sigma^2+\sigma+1}+Y^{\sigma^2+\sigma}+Y^{\sigma^2}+Y^{\sigma}+1$ has no solution in $\F_{2^m}^2$. Then the differential uniformity of $F$ is $2^d$.
\end{theorem}

We observe that following:
\begin{proposition} \label{prop:LK}
    All functions in Theorem~\ref{thm:li1} are linearly equivalent to a function from the family introduced in Theorem~\ref{thm_main} for the choices $a=b=1$, $c=0$. The conditions in Theorem~\ref{thm:li1} can be replaced by the single condition that $X^{\sigma^2+\sigma+1}+X^{\sigma^2}+1$ has no roots in $\F_{2^m}$.
\end{proposition}
\begin{proof}
    Take the function $F(x,y,z)=(x^{\sigma+1}+x^{\sigma}z+yz^{\sigma},x^{\sigma}z+y^{\sigma +1},xy^{\sigma}+y^{\sigma}z+z^{\sigma+1})$ from Theorem~\ref{thm:li1}. Changing the second and third component as well as renaming $y \leftrightarrow z$ yields exactly a function from the family introduced in Theorem~\ref{thm_main} for the choices $a=b=1$, $c=0$ as claimed.
\end{proof}

This family has very recently been generalized by Bartoli and Stanica to a family of the form 
\[F(x,y,z)=(x^{\sigma+1}+ax^\sigma z+yz^\sigma,x^\sigma z+y^{\sigma+1},xy^\sigma+ay^\sigma z+z^{\sigma+1}).\]
This family reduces to the original family found by Li and Kaleyski in the case $a=1$; and is contained in our family via linear equivalence from Theorem~\ref{thm_main} using the parameter choices $a=1$, $c=0$ in the same way described in Proposition~\ref{prop:LK}.

The second family found by Li and Kaleyski is:
\begin{theorem}[{\cite[Theorem 9]{li2023two}}] \label{thm:li2}
      Let $q=2^m$, $\gcd(k,m)=d$, $\sigma=2^k$ and define the following function on $\F_q^3$: $$F(x,y,z)=(x^{\sigma+1}+xy^{\sigma}+yz^{\sigma},xy^{\sigma}+y^{\sigma +1},x^{\sigma}z+y^{\sigma}z+y^{\sigma+1}).$$
    Assume that the polynomials  $X^{\sigma^2+\sigma+1}+X+1$, $X^{\sigma^2+\sigma+1}+X^{\sigma^2}+1$ and $X^{\sigma^2+\sigma+1}+X^{\sigma+1}+X^{\sigma}+X+1$ have no roots in $\F_{2^m}$ and $X^{\sigma^2+\sigma+1}+X^{\sigma+1}Y^{\sigma^2}+XY^{\sigma}+X^{\sigma}Y^{\sigma^2}+X^{\sigma^2}Y+X^{\sigma^2}+Y^{\sigma^2+\sigma+1}+Y^{\sigma^2+1}+Y^{\sigma^2+\sigma}+Y^{\sigma^2}+1$ has no solution in $\F_{2^m}^2$. Then the differential uniformity of $F$ is $2^d$.
\end{theorem}
Again, we observe that following:
\begin{proposition}
    All functions in Theorem~\ref{thm:li2} are linearly equivalent to a function from the family introduced in Theorem~\ref{thm_main} for the choices $a=c=1$, $b=0$. The conditions in Theorem~\ref{thm:li2} can be replaced by the single condition that $X^{\sigma^2+\sigma+1}+X^{\sigma}+1$ has no roots in $\F_{2^m}$.
\end{proposition}
\begin{proof}
    Take the function $F(x,y,z)=(x^{\sigma+1}+xy^{\sigma}+yz^{\sigma},xy^{\sigma}+y^{\sigma +1},x^{\sigma}z+y^{\sigma}z+y^{\sigma+1})$ from Theorem~\ref{thm:li2}. A cyclic shift of the components and renaming $x \leftrightarrow y$ yields exactly a function from the family introduced in Theorem~\ref{thm_main} for the choices $a=c=1$, $b=0$ as claimed.
\end{proof}
As we explained previously, our family yields for instance APN functions on $\F_{2}^{3m}$ for any $m>1$, while that is not the case for the families found by Li and Kaleyski (see~\cite[Proposition 4]{li2023two}). This means that our family not only unifies the two families by Li and Kaleyski, but provably yields new examples outside of these families as well (even up to CCZ-equivalence). Given that our family allows a large set of possible parameters $a,b,c \in \F_q$, we expect that our family actually contains many examples that are inequivalent to known APN functions.

The last family of trivariate APN functions (recently found by Bartoli and Stanica) is:
\begin{theorem}\label{thm:BS}{\cite[Theorem 3.2.]{bartoli2026infinite}}
      Let $q=2^m$, $\gcd(k,m)=d$, $\sigma=2^k$ and define the following function on $\F_q^3$: 
      \[F(x,y,z)=(x^{q+1}+axy^q+yz^q,xy^q+z^{q+1},x^qz+y^{q+1}+ay^qz).\]
      Then $F$ is APN if and only if $X^{\sigma^2+\sigma+1}+aX^{\sigma^2+\sigma}+1$ has no roots in $\F_q$.
\end{theorem}
Again, we can identify this family as a special case of our family, both generalizing and simplifying proofs considerably.
\begin{proposition}
    The family in Theorem~\ref{thm:BS} is linearly equivalent to a function in the family from Theorem~\ref{thm_main} for $a=1$ and $b=0$.
\end{proposition}
\begin{proof}
    Take the function from Theorem~\ref{thm_main} and rename the variables $x \leftrightarrow y$. Rotating the components yields then exactly a special case of the family in Theorem~\ref{thm_main}.
\end{proof}

\section{Inequivalence to Gold APN functions}
Other known infinite families of APN functions that yield examples in $\F_2^{n}$ with $n=3k$ for any choice of $k$ are monomials. By a result of Yoshiara, the only monomials that are possibly CCZ-equivalent to the family of APN functions in Theorem~\ref{thm_main} are in the family of Gold functions $F(x)=x^{2^i+1}$ for $\gcd(i,n)=1$~\cite{yoshiara2016equivalences}. In this section, we prove that the APN functions in Theorem~\ref{thm_main} are inequivalent to the Gold APN functions. The proof follows the same lines as a similar odd-characteristic proof given by the authors in~\cite{comm_paper} by proving that the linear automorphism groups of the respective functions cannot be conjugate in the general linear group. We want to note that for none of the trivariate APN functions contained in our family (as described in Section~\ref{sec:comparison}) such a theoretical inequivalence result was provided.

We recall some facts and definitions needed. Firstly, a result of Yoshiara~\cite{yoshiara2012equivalences} together with an observation in~\cite{kaspers2021number} yields that the APN functions in Theorem~\ref{thm_main} are CCZ-equivalent to Gold APN functions if and only if they are linearly equivalent to them. 
\begin{definition}
    Let $F \colon \F_2^n\rightarrow \F_2^n$ be a function. We define its linear automorphism group by
    \[\Aut_L(F)=\{(L_1,L_2)\in \GL(n,2)^2 \colon L_1 \circ F=F\circ L_2\}.\]
\end{definition}
The following well known and simple criterion can then be used to prove inequivalence of functions, see e.g.~\cite{gologlu2025equivalences}.
\begin{proposition} \label{prop:autgroups}
    Let $F,G \colon \F_2^n \rightarrow \F_2^n$ be linearly equivalent functions via $L_1\circ F=G\circ L_2$. Then $\Aut_L(G)=(L_1,L_2)\Aut_L(F)(L_1^{-1},L_2^{-1})$.
\end{proposition}
The linear automorphism groups of the Gold APN functions were determined by Kaspers and Zhou.
\begin{theorem}[{\cite[Theorem 4.1.]{kaspers2022lower}}] \label{thm:autgold}
    Let $G \colon \F_{2^n} \rightarrow \F_{2^n}$ be defined by $G(x)=x^{2^i+1}$ with $1\leq i <n$, $\gcd(i,n)=1$ and $n \geq 5$. Then $L_1 \circ G =G\circ L_2$ for $L_1,L_2 \in \GL(n,2)$ if and only if $L_1=a^{2^i+1}x^{2^k}$ and $L_2=ax^{2^k}$ for some $a \in \F_{2^n}^\times$ and $1\leq k <n$. In particular, the linear automorphism group is isomorphic to $\GammaL(1,2^n)$.
\end{theorem}

We now prove that the linear automorphism groups of the functions $F$ defined in~\ref{thm_main} cannot be conjugate to the ones of Gold APN functions. Note that we already identified some automorphisms of $F$ in Eq.~\eqref{eq:aut}. We thus have: 
\begin{lemma} \label{lem:subgroup}
    Let $F$ be as in Theorem~\ref{thm_main} with associated field automorphism 
	$\sigma=2^k$. Let 
	\[
	U_k=\left\{\left(\diag(a^{2^k+1},a^{2^k+1},a^{2^k+1}),\diag(a,a,a)\right) : a \in \F_{2^m}^\times \right\}\leq \GL(3,2^m)^2.
	\]
	Then $U_k \leq \Aut_L(F)$.
\end{lemma}

We also need some lemmas and known results:
\begin{lemma}[Zsigmondy's theorem]\cite[Chapter IX., Theorem 8.3.]{HuppertII} \label{thm:zsyg}
    Let $n>1$, $n\neq 6$ be an integer. Then there exists a prime $p$ that divides $2^n-1$ but not $2^k-1$ for all $1\leq k <n$.
\end{lemma}

\begin{lemma} \label{lem:normalizer}
    Let $S \leq \GL(3m,2)$ be a subgroup of the general linear group corresponding to the mappings $S=\{x\mapsto ax \colon a \in H\}$ where $H \leq \F_{2^m}^*$ is a subgroup of $\F_{2^m}^*$ such that $H$ is not contained in a proper subfield of $\F_{2^m}$. Then the normalizer $N_{\GL(3m,2)}(S)$ of $S$ in $\GL(3m,2)$ is exactly $N_{\GL(3m,2)}(S)=\GammaL(3,2^m)$. 
\end{lemma}
\begin{proof}
    Let $L \in \GL(3m,2)$. Write $L$ as a linearized polynomial, i.e., $L=\sum_{i=0}^{3m-1} c_ix^{2^i}$ for $c_i \in \F_{2^{3m}}$. We have $L \in N_{\GL(3m,2)}(S)$ if and only if $LS=SL$. So consider $A_a\in S$ defined by $x \mapsto ax$. Then 
    \[LA_a=\sum_{i=0}^{3m-1} c_i(ax)^{2^i}=\sum_{i=0}^{3m-1} c_ia^{2^i}x^{2^i}\]
    and
    \[A_{a'}L=a'\sum_{i=0}^{3m-1} c_ix^{2^i}=\sum_{i=0}^{3m-1} c_ia'x^{2^i}.\]
    So if $LA_a=A_{a'}L$ then $a'=a^{2^i}$ for all $i$ such that $c_i\neq 0$. Assume $c_k \neq 0$, so $a'=a^{2^k}$. Choose $a\in H$ in a way that it is not contained in a proper subfield of $\F_{2^m}$, so $a^{2^i}=a$ if and only if $i$ is a multiple of $m$. This implies only $c_k,c_{k+m},c_{k+2m}$ can be nonzero (where we take indices modulo $3m$). So $L=c_kx^{2^k}+c_{k+m}x^{2^{k+m}}+c_{k+2m}x^{2^{k+2m}}$. We conclude the proof by noticing that the set of invertible mappings of this form is exactly $\GammaL(3,2^m)$. 
\end{proof}

\begin{theorem}
    Let $q=2^m$, $m>2$, $m\neq 6$, $F=(F_1,F_2,F_3)\colon \F_{q}^3 \rightarrow \F_q^3$ be as in Theorem~\ref{thm_main} with associated field automorphism $\sigma=2^k$, $1\leq k <m$. Then $F$ is CCZ-inequivalent to any Gold APN function.
\end{theorem}
\begin{proof}
    We first consider a special case.
    \begin{claim}
        There are no $L_1,L_2 \in \GammaL(3,2^m)$ such that $L_1 \circ F = G \circ L_2$  for a Gold APN function $G$.
    \end{claim}
    
    \begin{proof}[Proof of claim.]
        Assume first $L_1,L_2 \in \GL(3,2^m)$.
        Such an equivalence exists if and only if there exist ordered $\F_q$-bases $\{u_1,v_1,w_1\}$ and $\{u_2,v_2,w_2\}$ of $\F_{q^3}$ such that $G(u_1x+v_1y+w_1z)=u_2F_1(x,y,z)+v_2F_2(x,y,z)+w_2F_3(x,y,z)$ since a transformation with an element in $\GL(3,2^m)$ corresponds exactly to a change of basis. Expanding $G(u_1x+v_1y+w_1z)$ for $G(x)=x^{2^i+1}$ yields a non-vanishing term of the form $u_1v_1^{2^i}xy^{2^i}$. No matter the choice of $k$, such a term does not exist in $F_1,F_2,F_3$ in Theorem~\ref{thm_main}. This proves the claim for $L_1,L_2 \in \GL(3,2^m)$.
        Now assume an equivalent of the form $L_1,L_2 \in \GammaL(3,2^m)$ exists. It is immediately clear by expanding again that the associated field automorphisms to $L_1,L_2$ have to be identical. Since all field automorphisms $(x \mapsto x^{2^l},x \mapsto x^{2^l})$ are contained in $\Aut_L(G)$ (see Theorem~\ref{thm:autgold}), we then necessarily also have an equivalence of the form $L_1,L_2 \in \GL(3,2^m)$, yielding the desired contradiction.
    \end{proof}
    Now assume $F$ is linearly equivalent to a Gold APN function $G$ and let $\sigma=2^k$ be the field automorphism associated to $F$. By Proposition~\ref{prop:autgroups}, there then exists $g=(L_1,L_2) \in \GL(3m,2)^2$ such that $g\Aut_L(F)g^{-1}=\Aut_L(G)$. Now let $p$ be a prime divisor of $2^m-1$ that does not divide $2^l-1$ for any $1 \leq l<m$. The existence of such a $p$ under our conditions is guaranteed by Theorem~\ref{thm:zsyg}, and it is immediate that $p>m$. 
Consider the subgroup $U_k\leq \Aut_L(F)$ defined in Lemma~\ref{lem:subgroup}. $U_k$ is cyclic of order $2^m-1$, so it has a unique Sylow $p$-subgroup, say $S \leq U_k$. Because $g S g^{-1} \leq \Aut_L(G)$ is a $p$-group of the same size, it must be the unique Sylow $p$-subgroup of $\Aut_L(G)$, which we will denote by $S_G$. Let $P$ be the unique Sylow $p$-subgroup of $\F_{2^m}^*$. By Lemma~\ref{lem:subgroup} and Theorem~\ref{thm:autgold}, the second components of the pairs in both $S$ and $S_G$ correspond exactly to the set of scalar multiplications $S_2 = \{x \mapsto ax \mid a \in P\}$. Because $g S g^{-1} = S_G$, $L_2$ must normalize $S_2$. That is, $L_2 S_2 L_2^{-1} = S_2$. Since $p$ does not divide $2^l-1$ for $l < m$, $P$ is not contained in a proper subfield of $\F_{2^m}$. By Lemma~\ref{lem:normalizer}, this means $L_2 \in \GammaL(3,2^m)$.	
Similarly, the first component of $S_G$ is exactly the set of scalar multiplications by $a^{2^i+1}$ for $a \in P$. Since $F$ and $G$ are APN, $\gcd(i, 3m)=1$ and $\gcd(k,m)=1$. Because the order of $2$ modulo $p$ is exactly $m$, this implies $p$ cannot divide $2^i+1$ or $2^k+1$. Thus, the maps $a \mapsto a^{2^i+1}$ and $a \mapsto a^{2^k+1}$ are bijective on $P$, and the first component of $S_G$ and $S$ are also exactly the set $S_2$. This forces $L_1 \in N_{\GL(3m,2)}(S_2) = \GammaL(3,2^m)$, using again Lemma~\ref{lem:normalizer}.
		So $L_1,L_2 \in \GammaL(3,2^m)$ which, together with the claim, yields the result. 
\end{proof}

\section{Conclusion and further notes}
We constructed a new family of APN functions, and in particular, APN permutations. This family unifies several known constructions. The proof techniques we used differed considerably from the ones used in previous papers and indicate a more general way forward by connecting the constructions of highly nonlinear functions to the constructions of certain permutations of the projective space. A natural open question is:
\begin{problem}
	Find new constructions and classifications of semi-quadratic homogeneous functions that lead to bijections of the projective space.
\end{problem}
As detailed in this paper (cf. Theorem~\ref{thm:unif}) such functions will lead to the construction of highly nonlinear functions. There are further open questions. As outlined in Remark~\ref{rem:mdeven}, if $m/d$ is even, the Walsh spectra of the APN functions we constructed in Theorem~\ref{thm_main} are completely known. In the case of the APN permutations (i.e. for odd $m$), this remains an open problem.
\begin{problem}
	Determine the Walsh spectrum of the family of APN permutations in Theorem~\ref{thm_main}.
\end{problem}

\bibliographystyle{amsplain}
\bibliography{semifields} 

\bigskip
\hrule
\bigskip
\end{document}